\pdfoutput=1
\RequirePackage{ifpdf}
\ifpdf % We are running pdfTeX in pdf mode
\documentclass[pdftex]{sigma}
\else
\documentclass{sigma}
\fi

\numberwithin{equation}{section}

\newtheorem{Theorem}{Theorem}[section]
\newtheorem{Lemma}[Theorem]{Lemma}
\newtheorem{Proposition}[Theorem]{Proposition}
{ \theoremstyle{definition}
\newtheorem{Definition}[Theorem]{Definition}
\newtheorem{Note}[Theorem]{Note}
\newtheorem{Remark}[Theorem]{Remark} }

\begin{document}

\allowdisplaybreaks

\renewcommand{\thefootnote}{$\star$}

\newcommand{\arXivNumber}{1601.06898}

\renewcommand{\PaperNumber}{075}

\FirstPageHeading

\ShortArticleName{Orthogonal Polynomials Associated with Complementary Chain Sequences}

\ArticleName{Orthogonal Polynomials Associated\\ with Complementary Chain Sequences\footnote{This paper is a~contribution to the Special Issue on Orthogonal Polynomials, Special Functions and Applications. The full collection is available at \href{http://www.emis.de/journals/SIGMA/OPSFA2015.html}{http://www.emis.de/journals/SIGMA/OPSFA2015.html}}}

\Author{Kiran Kumar BEHERA~$^\dag$, A.~SRI RANGA~$^\ddag$ and A.~SWAMINATHAN~$^\dag$}

\AuthorNameForHeading{K.K.~Behera, A.~Sri Ranga and A.~Swaminathan}

\Address{$^\dag$~Department of Mathematics, Indian Institute of Technology Roorkee,\\
\hphantom{$^\dag$}~Uttarakhand-247667, India}
\EmailD{\href{mailto:krn.behera@gmail.com}{krn.behera@gmail.com}, \href{mailto:mathswami@gmail.com}{mathswami@gmail.com}}

\Address{$^\ddag$~Departamento de Matem\'{a}tica Aplicada, IBILCE, UNESP-Univ. Estadual Paulista,\\
\hphantom{$^\ddag$}~15054-000, S\~{a}o Jos\'{e} do Rio Preto, SP, Brazil}
\EmailD{\href{mailto:ranga@ibilce.unesp.br}{ranga@ibilce.unesp.br}}

\ArticleDates{Received March 17, 2016, in f\/inal form July 22, 2016; Published online July 27, 2016}

\Abstract{Using the minimal parameter sequence of a given chain sequence, we introduce the concept of complementary chain sequences, which we view as perturbations of chain sequences. Using the relation between these complementary chain sequences and the corresponding Verblunsky coef\/f\/icients, the para-orthogonal polynomials and the associated Szeg\H{o} polynomials are analyzed. Two illustrations, one involving Gaussian hypergeometric functions
and the other involving Carath\'{e}odory functions are also provided. A connection between these two illustrations by means of complementary chain sequences is also observed.}


\Classification{42C05; 33C45; 30B70}

\renewcommand{\thefootnote}{\arabic{footnote}}
\setcounter{footnote}{0}

\section{Preliminaries on Szeg\H{o} polynomials}\label{sec-complement_chain_SzegoPoly}

The Szeg\H{o} polynomials $\{\Phi_n\}$, also referred to as orthogonal polynomials on the unit circle (OPUC), enjoy the orthogonality property
\begin{gather*}
\int_{\partial\mathbb{D}}(\bar{z})^j\Phi_n(z)d\mu(z)= \int_{\partial\mathbb{D}}(z)^{-j}\Phi_n(z)d\mu(z)=0 \qquad\mbox{for}\quad
j=0,1,\dots,n-1,\quad n\geq1.
\end{gather*}
Here $\mu(z)=\mu(e^{i\theta})$ is a nontrivial measure def\/ined on the unit circle $\partial{\mathbb{D}}=\{z=e^{i\theta}\colon 0\leq\theta\leq2\pi\}$. Denoting the orthonormal Szeg\H{o} polynomials by $\phi_n(z)=\chi_n\Phi_{n}(z)$, we also have the equivalent def\/inition
\begin{gather*}
\int_{\partial\mathbb{D}}\phi_{n}(z) \overline{\phi_m(z)}d\mu(z)= \delta_{m,n}.
\end{gather*}
Further, def\/ining the moments $\mu_n=\int_{\partial{\mathbb{D}}} e^{-in\theta}d\mu(\theta)$, $n=0,\pm1,\dots$, where $\mu_{-n}=\bar{\mu}_n$, we have
\begin{gather*}
\int_{\partial\mathbb{D}}(\bar{z})^n \Phi_n(z) d\mu(z)= \frac{\Delta_n}{\Delta_{n-1}}\neq0, \qquad n=0,1,\dots.
\end{gather*}
Here $\Delta_n=\det\{\mu_{i-j}\}_{i,j=0}^{n}$ are the associated Toeplitz matrices with $\Delta_{-1}=1$.

The monic Szeg\H{o} polynomials satisfy the f\/irst order recurrence relations
\begin{gather*}
\Phi_{n}(z) =z\Phi_{n-1}(z)-\bar{\alpha}_{n-1}\Phi_{n-1}^{*}(z), \qquad
\Phi_{n}^{*}(z) =-\alpha_{n-1}z\Phi_{n-1}(z)+\Phi_{n-1}^{*}(z), \qquad n\geq1,
\end{gather*}
where $\Phi_{n}^{*}(z)=z^{n}\overline{\Phi_{n}(1/\bar{z})}$. The complex numbers $\alpha_{n-1}=-\overline{\Phi_n(0)}$ are called the Verblunsky coef\/f\/icients~\cite{Simon1}. The Verblunsky coef\/f\/icients completely characterize the Szeg\H{o} polynomials in the sense that any sequence $\{\alpha_{n-1}\}_{n=1}^{\infty}$ lying within the unit circle gives rise to a unique probability measure $\mu(z)$ which leads to a unique sequence of Szeg\H{o} polynomials. The above result, called the Verblunsky theorem in~\cite{Simon1}, is the analogue of Favard's theorem on the real line. Conversely, algorithms exist in the literature that extracts these coef\/f\/icients from any given Szeg\H{o} system of orthogonal polynomials. Notable among them are the Schur algorithm, the Levinson algorithm and their modif\/ied versions given in~\cite{DG,JNT-Surv,Wall}.

The Szeg\H{o} polynomials also satisfy the three term recurrence relation
\begin{gather}\label{eqn:szego_ttrr}
\Phi_{n+1}(z)=\left(\frac{\Phi_{n+1}(0)}{\Phi_n(0)}+z\right) \Phi_{n}(z)-\frac{(1-|\Phi_n(0)|^2)\Phi_{n+1}(0)}{\Phi_n(0)} z\Phi_{n-1}(z), \qquad n\geq1,
\end{gather}
with $\Phi_0(z)=1$ and $\Phi_1(z)=z+\Phi_1(0)$. Note that if $\Phi_n(0) = 0$, $n\geq1$, then the three term recurrence relation ceases to exist. In such a case, $\Phi_n(z)=z^n$, which is given as the free case in \cite[p.~85]{Simon1}. Denoting
\begin{gather*}
\eta_{n+1}=\frac{\Phi_{n+1}(0)}{\Phi_n(0)} \qquad\mbox{and}\qquad \rho_{n+1}=\frac{\big(1-|\Phi_n(0)|^2\big)\Phi_{n+1}(0)}{\Phi_n(0)}, \qquad n\geq1,
\end{gather*}
the following expressions are easily obtained from~\eqref{eqn:szego_ttrr}:
\begin{gather*}
\Phi_{n+1}(0)= \frac{\Phi_n(0)}{1-|\Phi_n(0)|^2} \frac{\int_{\partial\mathbb{D}} z \Phi_{n}(z)d\mu(z)} {\int_{\partial\mathbb{D}}z \Phi_{n-1}(z)d\mu(z)},
\qquad 1-|\Phi_n(0)|^2= \frac{\int_{\partial\mathbb{D}} z^{-n} \Phi_{n}(z)d\mu(z)} {\int_{\partial\mathbb{D}} z^{-(n-1)} \Phi_{n-1}(z)d\mu(z)}
\end{gather*}
and
\begin{gather}\label{eqn:szego_norm}
\chi_{n}^{-2}= \frac{\rho_{2}\rho_3\cdots\rho_{n+1}} {\eta_2\eta_3\cdots\eta_{n+1}}\mu_0= \mu_0\big(1-|\Phi_1(0)|^2\big)\big(1-|\Phi_2(0)|^2\big)\cdots \big(1-|\Phi_n(0)|^2\big).
\end{gather}
For early developments on the subject, we refer to the monographs \cite{Freud,Gero,Szego}. For a compendium of modern research in the area as well as historical notes, we refer to~\cite{Simon1,Simon2}.

In order to develop a quadrature formula on the unit circle, Jones et al.~\cite{JNT-Mom} introduced the para-orthogonal polynomials which vanish only on the unit circle and, for $z,\omega_n\in\mathbb{C}$ with $|\omega_n|=1$, have the representation
\begin{gather*}
\mathcal{X}_{n}(z,\omega_n)=\Phi_n(z)+ \omega_n\Phi^{*}_{n}(z), \qquad n\geq1.
\end{gather*}
The para-orthogonal polynomials satisfy the properties
\begin{gather*}
\langle \mathcal{X}_n, z^m\rangle=0, \qquad m=1,2,\dots,n-1, \qquad \langle \mathcal{X}_n, 1\rangle\neq0, \qquad \langle \mathcal{X}_n, z^n\rangle\neq0,
\end{gather*}
which are termed as def\/iciency in the orthogonality of these para-orthogonal polynomials. In recent years, these para-orthogonal polynomials have been linked to kernel polynomials $K_n(z,\omega)$, see~\cite{Ranga_chn,Golli,Wong}. The kernel polynomials $K_n(z,\omega)$ satisfy the Christof\/fel--Darboux formula
\begin{gather*}
K_n(z,\omega)=\sum_{k=0}^{n}\phi_k(z)\overline{\phi_k(\omega)}= \frac{\phi^{*}_{n+1}(z)\overline{\phi^{*}_{n+1}(\omega)}-\phi_{n+1}(z)\overline{\phi_{n+1}(\omega)}}{1-z\bar{\omega}}.
\end{gather*}
Denoting $\tau_n(\omega)=\Phi_{n}(\omega)/\Phi^{*}_{n}(\omega)$, for $n\geq1$, the monic kernel polynomials related to the Szeg\H{o} polynomials are given by
\begin{gather}\label{eqn:szego_monic_kernel}
P_{n}(\omega;z)=\frac{z\Phi_{n}(z)- \omega\tau_{n}(\omega)\Phi^{*}_{n}(z)}{z-\omega}, \qquad n\geq1,
\end{gather}
and are shown in \cite{Ranga_chn} to satisfy a three term recurrence relation of the form
\begin{gather*}%\label{eqn:szego_kernel_ttrr}
P_{n+1}(\omega;z)=[z+b_{n+1}(\omega)]P_{n}(\omega;z)-a_{n+1} (\omega)zP_{n-1}(\omega;z), \qquad n\geq1,
\end{gather*}
where
\begin{gather*}
b_n(\omega)=\frac{\tau_n(\omega)}{\tau_{n-1}(\omega)}, \qquad a_{n+1}=[1+\tau_n(\omega)\alpha_{n-1}] \big[1-\overline{\omega\tau_n(\omega)\alpha_n}\big]\omega,
\qquad n\geq1.
\end{gather*}
The polynomials $P_{n}(\omega;z)$ are $\overline{\tau_n(w)}$-invariant sequences of polynomials which can be easily verif\/ied from~\eqref{eqn:szego_monic_kernel}. Note that a sequence of polynomials $\{\mathcal{Y}_{n}\}$ is called $\tau_n$-invariant if~\cite{JNT-Mom}
\begin{gather*}
\mathcal{Y}^{*}_{n}(z)=\tau_n\mathcal{Y}_{n}(z),\qquad n\geq1.
\end{gather*}
An important concept that is used in the sequel is the theory of chain sequences. We give a~brief introduction to chain sequences and then illustrate the role played by them in the theory of OPUC.

A sequence $\{d_n\}_{n=1}^{\infty}$ which satisf\/ies
\begin{gather*}
d_{n}=(1-g_{n-1})g_n, \qquad n\geq1,
\end{gather*}
is called a positive chain sequence \cite{Chihara} (see also \cite[Section~7.2]{Ismail}). Here $\{g_n\}_{n=0}^{\infty}$, called the parameter sequence is such that $0\leq g_0<1$, $0<g_n<1$ for $n\geq1$. This is a stronger condition than the one used in~\cite{Wall}, in which $d_n$ is also allowed to be zero. The parameter sequen\-ce~$\{g_n\}_{n=0}^{\infty}$ is called a minimal parameter sequence and denoted by $\{m_n\}_{n=0}^{\infty}$ if $m_0=0$. Every chain sequence has a~minimal parameter sequence \cite[pp.~91--92]{Chihara}. Further, for a f\/ixed chain sequence $\{d_n\}_{n\geq1}$, let $\mathcal{G}$ be the set of all parameter sequen\-ces~$\{g_k\}$ of $\{d_n\}_{n\geq1}$. Let the sequen\-ce~$\{M_n\}_{n=0}^{\infty}$ be def\/ined by
\begin{gather*}
M_n=\inf\{g_n,\,\mbox{for each } n,\,\{g_k\}\in\mathcal{G}\}, \qquad n\geq0,
\end{gather*}
where $\inf$ is inf\/imum of the set. Then, $\{M_n\}$ is called the maximal parameter sequence of $\{d_n\}$.

The role of chain sequences in the study of orthogonal polynomials on the real line is well known. Similarly, a positive chain sequence $\{d_n\}$ appears
in the three term recurrence relation for the polynomials~$R_n(z)$, that turn out to be scaled versions of the (kernel) polynomials $P_n(1,z)$, namely,
\begin{gather}\label{eqn:POP_ttrr_c_n-d_n}
R_{n+1}(z)=[(1+ic_{n+1})z+(1-ic_{n+1})]R_n(z)-4d_{n+1}zR_{n-1}(z), \qquad n\geq1,
\end{gather}
with $R_0(z)=1$ and $R_1(z)= (1+ic_1)z+(1-ic_1)$. It is indeed shown in \cite{Ranga_chn}, that
\begin{gather*}
R_n(z)=\frac{\prod\limits_{j=0}^{n-1}[1-\tau_j\alpha_j]} {\prod\limits_{j=0}^{n-1}[1-\operatorname{Re}(\tau_j\alpha_j)]} P_{n}(1;z),
\end{gather*}
where
\begin{gather*}
 \tau_{j} = \tau_j(1)= \prod_{k=1}^{j} \frac{1-ic_k}{1+ic_k}, \qquad j \geq 1,
\end{gather*}
on condition that
\begin{gather*}
c_n=\frac{-\operatorname{Im}(\tau_{n-1}\alpha_{n-1})} {1-\operatorname{Re}(\tau_{n-1}\alpha_{n-1})}
\qquad \mbox{and}\qquad d_{n+1}= (1-g_{n}) g_{n+1}, \qquad n\geq1,
\end{gather*}
is a chain sequence with parameter sequence
\begin{gather*}
 g_{n} = \frac{1}{2} \frac{|1-\tau_{n-1}\alpha_{n-1}|^2} {[1-\operatorname{Re}(\tau_{n-1}\alpha_{n-1})]}, \qquad n \geq 1.
\end{gather*}
It is also not dif\/f\/icult to verify that in this case $R_n(z)$ has $r_{n,n}=\prod\limits_{k=1}^{n}(1+ic_k)$ as the leading coef\/f\/icient and $r_{n,0}=\bar{r}_{n,n}=\prod\limits_{k=1}^{n}(1-ic_k)$ as the constant term.

It is known that $\{R_n(z)\}$ can be used to obtain a sequence of OPUC~\cite{Swami,Ranga_fav,Ranga_chn}, with respect to the measure $\mu(z)$ and having the shifted sequence $\{\alpha_{n-1}\}_{n=1}^{\infty}$ as the Verblunsky coef\/f\/icients. A~further interesting fact is that the above parameter sequence
$\{g_{n+1}\}_{n=0}^{\infty}$ is such that $g_{1} = (1-\epsilon)M_{1}$, ($0\leq\epsilon<1$), where $\{M_{n+1}\}_{n=0}^{\infty}$ is the maximal parameter sequence of $\{d_{n+1}\}_{n=1}^{\infty}$ and that $\epsilon$, is the size of the pure point at $z=1$ in the probability measure $\mu(z)$ associated with the
Verblunsky coef\/f\/icients $\{\alpha_{n-1}\}_{n=1}^{\infty}$. This means, if the measure does not have a pure point at $z=1$ then $\{g_{n+1}\}_{n=0}^{\infty}$ is the maximal parameter sequence of $\{d_{n+1}\}_{n=1}^{\infty}$.

Consider now the Uvarov transformation of the measure $\mu(z)$, \cite[p.~11]{Ranga_chn},
\begin{gather*}%\label{eqn:uvarov transformation of measure}
 \int_{\partial\mathbb{D}} f(z)d\mu^{(t)}(z) = \frac{(1-t)}{1-\epsilon}\int_{\partial\mathbb{D}} f(z) d\mu(z) + \frac{t-\epsilon}{1-\epsilon} f(1),
\end{gather*}
so that $\mu^{(t)}(z)$ has a jump $t$, $0\leq t<1$, at $z=1$. These measures $\mu^{(t)}(z)$ are associated with the positive chain sequence $\{d_n\}_{n=1}^{\infty}$ obtained from $\{d_{n+1}\}_{n=1}^{\infty}$ by including the additional term $d_1 = (1-t)M_1$. We also denote the generalized sequence of Verblunsky coef\/f\/icients associated with $\mu^{(t)}(z)$ by $\big\{\alpha_{n-1}^{(t)}\big\}_{n=1}^{\infty}$.

\begin{Note}
Since the measure $\mu^{(t)}(z)$ has a parameter `$t$', the notations for the polyno\-mials~$R_n(z)$ and the sequences $\{c_n\}$, $\{d_n\}$ should have involved a $'t'$. However, it has been proved~\cite[p.~7]{Ranga_chn}, that the kernel polynomials $P_n(1;z)$ and hence $R_n(z)$ as well as the sequen\-ces~$\{c_n\}$ and~$\{d_n\}$ are independent of `$t$' and so their notations are devoid of~`$t$'. But the minimal parameters depend on $d_1$ and this has been ref\/lected in the notation~$m_n^{(t)}$.
\end{Note}

As shown in \cite[Theorem~1.1]{Ranga_chn}, $\mu^{(t)}(z)$ can also be given by
\begin{align}\label{eqn:relation between measure with jump t and measure with zero jump}
 \int_{\partial\mathbb{D}} f(z)d\mu^{(t)}(z) =
 (1-t)\int_{\partial\mathbb{D}} f(z) d\mu^{(0)}(z) +
 t f(1).
\end{align}
As is obvious from the notation, $\mu^{(0)}(z)$ are the measures arising when $t=0$. This is the case when $d_1=M_1$, so that both the minimal and maximal parameter sequences coincide. This equality can also be interpreted as the measure having zero jump.

Further, the Verblunsky coef\/f\/icients $\alpha_{n-1}^{(t)}$ have the representation
\begin{gather}\label{eqn:gen_verb_coeff}
\alpha_{n-1}^{(t)} = \overline{\tau}_n \left[\frac{1-2m_n^{(t)}-ic_{n}} {1+ic_{n}}\right],\qquad n\geq 1.
\end{gather}
where $\big\{m_n^{(t)}\big\}$ is the minimal parameter sequence of the positive chain sequence $\{d_n\}_{n=1}^{\infty}$. The Szeg\H{o} polynomials corresponding to~\eqref{eqn:gen_verb_coeff} are \cite[Theorem~5.2]{Ranga_fav}
\begin{gather}\label{eqn:szego_from_gen_verb_coeff}
\Phi_n^{(t)}(z)=\frac{R_n(z)-2\big(1-m_n^{(t)}\big)R_{n-1}(z)}{\prod\limits_{k=1}^{n}(1+ic_k)},\qquad n\geq1.
\end{gather}

It can be verif\/ied from~\eqref{eqn:POP_ttrr_c_n-d_n} that if $c_k=0$, $k\geq0$, $\alpha_{n-1}$, $n \geq 1$, are all real. The $R_n(z)$ are then the singular predictor polynomials of the second kind given in~\cite{DG}. Indeed, if $c_n=0$, $n\geq1$, it can be easily shown from~\eqref{eqn:szego_from_gen_verb_coeff} that
\begin{gather*}
(z-1)R_n(z) = z\Phi_{n}^{(t)}(z) - \big(\Phi^{(t)}_{n}\big)^{*}(z).
\end{gather*}
We would like to mention here that the Szeg\H{o} polynomials, Verblunsky coef\/f\/icients and the related measure have also been obtained for the para-orthogonal polynomials that are an extension of singular predictor polynomials of f\/irst kind. See~\cite{Swami} for the details of these extensions and also for a survey of recent developments in the theory connecting chain sequences and OPUC.

The purpose of the present manuscript is to introduce a particular perturbation in the chain sequence $\{d_n\}$, called the complementary chain sequence,
and study its ef\/fect on the Verblunsky coef\/f\/icients of the corresponding Szeg\H{o} polynomials. The motivation for this follows from the fact that~\eqref{eqn:gen_verb_coeff} guarantees an explicit relation between the Verblunsky coef\/f\/icients and the minimal parameter sequence~$\big\{m_n^{(t)}\big\}$ of~$\{d_n\}$.

This manuscript is organized as follows. In Section~\ref{sec-complement_chain_basics} the concept of complementary chain sequences using the minimal parameter sequences is introduced. Using this concept, perturbations of Verblunsky coef\/f\/icients are studied. As an illustration of this concept, in Section~\ref{sec-complement_carat}, the Szeg\H{o} polynomials which characterizes the positive Perron--Carath\'{e}odory (PPC) fractions from a particular chain sequence are constructed. An interplay by these PPC fractions in f\/inding a relation between this chain sequence, its complementary chain sequence and their respective Carath\'{e}odry functions is obtained in this section. In Section~\ref{sec-complement_gauss_hyper}, another illustration of characterizing the
Szeg\H{o} polynomials using Gaussian hypergeometric functions is provided. For particular values, using complementary chain sequences, the corresponding Verblunsky coef\/f\/icients of these Szeg\H{o} polynomials are also shown to be perturbed Verblunsky coef\/f\/icients obtained earlier.

\section{Complementary chain sequences}\label{sec-complement_chain_basics}
As is obvious from the def\/inition of chain sequences, the minimal and maximal parameter sequences are uniquely def\/ined for any given chain sequence. Also, the chain sequence for which the minimal and maximal parameter sequences coincide, that is, $M_0=0$, has its own importance as illustrated in the previous section. Such a chain sequence is said to determine its parameters uniquely and is referred to as a single parameter positive chain sequence (SPPCS)~\cite{Swami}. By Wall's criteria for maximal parameter sequence~\cite[p.~82]{Wall}, this is equivalent to
\begin{gather}\label{eqn:wall_criteria_sppcs}
\sum_{n=1}^{\infty}\frac{m_1}{1-m_1} \cdot\frac{m_2}{1-m_2} \cdot\frac{m_3}{1-m_3} \cdots\frac{m_n}{1-m_n}=\infty.
\end{gather}
Thus, introducing a perturbation in the minimal parameters $m_n$ will lead to a uniquely def\/ined change in the chain sequence.

\begin{Definition} Suppose $\{d_n\}_{n=1}^{\infty}$ is a chain sequence with $\{m_n\}_{n=0}^{\infty}$ as its minimal parameter sequence. Let $\{k_n\}_{n=0}^{\infty}$ be another sequence given by $k_0=0$ and $k_n=1-m_n$ for $n\geq1$. Then the chain sequence $\{a_n\}_{n=1}^{\infty}$ having
$\{k_n\}_{n=0}^{\infty}$ as its minimal parameter sequence is called the complementary chain sequence of $\{d_n\}$.
\end{Definition}

Such chain sequences enjoy interesting relations like \cite[equation~(75.3)]{Wall}
\begin{gather*}
\frac{\sqrt{1+z}}{1+\frac{d_1z}{1+\frac{d_2z} {1+\frac{d_3z}{1+\ddots}}}} \cdot \frac{\sqrt{1+z}}{1+\frac{a_1z}{1+\frac{a_2z}{1+\frac{a_3z}{1+\ddots}}}}=1.
\end{gather*}
They also satisfy
\begin{gather*}
d_1-a_1=1-2k_1=2m_1-1
\end{gather*}
and
\begin{gather*}
d_n-a_n=\triangle{m_{n-1}}=-\nabla{k_n},\qquad n\geq2.
\end{gather*}
where $\triangle$ and $\nabla$ are the forward and backward dif\/ference operators respectively. Further of particular interest is the ratio of these two chain sequences given by
\begin{gather*}
\frac{d_1}{a_1}=\frac{m_1}{1-m_1}, \qquad \frac{d_n}{a_n}= \frac{k_{n-1}}{1-k_{n-1}} \frac{m_{n}}{1-m_{n}}, \qquad n\geq2.
\end{gather*}
This implies
\begin{gather}\label{eqn:min_par_comp_chn_seq}
\frac{m_n}{1-m_n}=\frac{d_n}{a_n}\frac{m_{n-1}}{1-m_{n-1}}=\cdots=\frac{d_{n}d_{n-1}\cdots d_{1}}{a_{n}a_{n-1}\cdots a_{1}},\qquad n\geq1.
\end{gather}
Substituting \eqref{eqn:min_par_comp_chn_seq} in \eqref{eqn:wall_criteria_sppcs}, we have the following lemma.

\begin{Lemma}\label{lem:criteria_sppcs_intermsof_comp_chn_seq} Let $\{d_n\}_{n=1}^{\infty}$ and $\{a_n\}_{n=1}^{\infty}$
be two complementary chain sequences of each other. Then $\{d_n\}_{n=1}^{\infty}$ will be a SPPCS if and only if
\begin{gather*}%\label{eqn:criteria_comp_chn_seq_sppcs}
\sum_{n=1}^{\infty}\prod_{j=1}^{n}\frac{d_{1}d_{2}\cdots d_{j}} {a_{1}a_{2}\cdots a_{j}}=\infty.
\end{gather*}
\end{Lemma}

\begin{Remark} The above lemma is useful while considering a chain sequence and its complementary chain sequence without using the information on the corresponding minimal parameters.
\end{Remark}

\begin{Lemma}\label{lem:criteria2_sppcs_intermsof_comp_chn_seq} Let $\{d_n\}_{n=1}^{\infty}$ and $\{a_n\}_{n=1}^{\infty}$ be two complementary chain sequences of each other. If $\{d_n\}_{n=1}^{\infty}$ is not a SPPCS, then $\{a_n\}_{n=1}^{\infty}$ is a SPPCS.
\end{Lemma}

\begin{proof} If $\{d_n\}_{n=1}^{\infty}$ is not a SPPCS then its minimal parameter sequence $\{m_n\}_{n=0}^{\infty}$ is such that
\begin{gather*}
\sum_{n=1}^{\infty}\prod_{j=1}^{n}\frac{m_{j}} {1-m_{j}} < \infty.
\end{gather*}
Hence, $\lim\limits_{n \to \infty} \prod\limits_{j=1}^{n} m_{j}/(1-m_{j}) = 0$, and we have
\begin{gather*}
 \sum_{n=1}^{\infty}\prod_{j=1}^{n}\frac{k_{j}}{1-k_{j}} = \sum_{n=1}^{\infty}\prod_{j=1}^{n}\frac{1-m_{j}}{m_{j}} = \infty.
\end{gather*}
Thus, concluding the proof of the lemma.
\end{proof}

\begin{Lemma}\label{lem:criteria_sppcs_1/2} Let $\{d_n\}_{n=1}^{\infty}$ be a chain sequence and $\{a_n\}_{n=1}^{\infty}$ be its complementary chain sequence with minimal parameter sequences $\{m_n\}_{n=0}^{\infty}$ and $\{k_n\}_{n=0}^{\infty}$
respectively.
\begin{enumerate}\itemsep=0pt
\item[--] If $0<m_n<1/2$, $n\geq1$, then $a_n$ is a SPPCS.\label{item2}
\item[--] If $1/2<m_n<1$, $n\geq1$, then $d_n$ is a SPPCS.\label{item3}
\end{enumerate}
\end{Lemma}

\begin{proof} Observe that if $0<m_n<1/2$, $k_n/(1-k_n)>1$ for all $n\geq1$. Similarly, $1/2<m_n<1$ implies $m_n/(1-m_n)>1$ for all $n\geq1$. The results now follow from~\eqref{eqn:wall_criteria_sppcs}.
\end{proof}

It is known that \cite[p.~79]{Wall} if $d_n\geq1/4$, $n\geq1$, every parameter sequence $\{g_n\}$, in particular the minimal parameter sequence $\{m_n\}$ of $\{d_n\}$ is non-decreasing. For the special case when $d_n=1/4$, $n\geq1$, $m_n\rightarrow1/2$ as $n\rightarrow\infty$. This implies $0<m_n<1/2$, $n\geq1$.
By Lemma~\ref{lem:criteria_sppcs_1/2}, $\{a_n\}$~is a SPPCS. In other words, the chain sequence complementary to the constant chain sequen\-ce~$\{1/4\}$ determines its parameters~$g_n$ uniquely, which are further given by
\begin{gather*}
g_0=0,\qquad g_n=\frac{n+2}{2(n+1)}, \qquad n\geq1.
\end{gather*}
Moreover, if $d_n\geq1/4$, there exist some $n\in\mathbb{N}$ such that $a_n<1/4\leq d_n$. Indeed,
\begin{gather*}
d_n=(1-m_{n-1})m_n\geq m_{n-1}(1-m_{n})=a_n, \qquad n\geq2,
\end{gather*}
with the sign of the dif\/ference of $d_1$ and $a_1$ depending on whether $m_1\in(0,1/2)$ or $(1/2,1)$. If $a_n\in(1/4, 1)$ for $n\geq1$, $k_n$ has to be
non-decreasing. This is a contradiction as $k_n=1-m_n$ for $n\geq1$.

The ef\/fect of complementary chain sequences in studying perturbation of Verblunsky coef\/f\/i\-cients given by~\eqref{eqn:gen_verb_coeff} has interesting consequences. In this context, we give the following result.

\begin{Theorem}\label{thm:comp_chn_seq_as_perturbation} Let $\{c_n\}_{n=1}^{\infty}$ and $\{d_{n+1}\}_{n=1}^{\infty}$ be, respectively, the real sequence and positive chain sequence as given in~\eqref{eqn:POP_ttrr_c_n-d_n}. Let $\big\{m_n^{(t)}\big\}_{n=0}^{\infty}$ be the minimal parameter sequence of the augmented positive chain sequence $\{d_n\}_{n=1}^{\infty}$, where $d_1 = (1-t)M_1$ and $\{M_{n+1}\}_{n=0}^{\infty}$ is the maximal parameter sequence of $\{d_{n+1}\}_{n=1}^{\infty}$. Let $\big\{k_n^{(t)}\big\}_{n=0}^{\infty}$ be the minimal parameter sequence of the positive chain sequence $\{a_n\}_{n=1}^{\infty}$ obtained as complementary to $\{d_n\}_{n=1}^{\infty}$. Set $\tau_n = \frac{1-ic_n}{1+ic_n} \tau_{n-1}$,
\begin{gather*}
 \alpha_{n-1}^{(t)} = \overline{\tau}_n \left[\frac{1-2m_n^{(t)}-ic_{n}}{1+ic_{n}}\right]
 \qquad \mbox{and} \qquad \beta_{n-1}^{(t)} = \overline{\tau}_n \left[\frac{1-2k_n^{(t)}-ic_{n}}{1+ic_{n}}\right],
\end{gather*}
for $n \geq 1$, with $\tau_{0} = 1$. Let $\mu^{(t)}(z)$ and $\nu^{(t)}(z)$ be, respectively, the probability measures having $\alpha_{n-1}^{(t)}$ and $\beta_{n-1}^{(t)}$ as the corresponding Verblunsky coefficients. Then the following can be stated:
\begin{enumerate}\itemsep=0pt
 \item[$1.$] For $0 < t < 1$, the measure $\mu^{(t)}(z)$ has a pure point of size $t$ at $z=1$, while $\nu^{(t)}(z)$ does not.

 \item[$2.$] $\beta_{n-1}^{(t)} = - \overline{\tau}_n \overline{\tau}_{n-1} \overline{\alpha}_{n-1}^{(t)}$, $n \geq 1$.

 \item[$3.$] For $n \geq 1$, if $c_n = (-1)^n c$, $c\in\mathbb{R}$, $\beta_{n-1}^{(t)}= - \frac{1-ic}{1+ic} \alpha_{n-1}^{(t)}$, $n\geq 1$.

 \item[$4.$] If $c_n = 0$, $n \geq 1$ then the Verblunsky coefficients, which are real, are such that $\beta_{n-1}^{(t)} = - \alpha_{n-1}^{(t)}$, $n\geq 1$.
\end{enumerate}
\end{Theorem}

\begin{proof}First we observe that $\alpha_{n-1}^{(t)}$ are the generalized Verblunsky coef\/f\/icients of the measure~$\mu^{(t)}(z)$ as given by~\eqref{eqn:relation between measure with jump t and measure with zero jump}. Consequently, for $0 < t < 1$ the probability measure $\mu^{(t)}(z)$ has a~pure point of size $t$ at $z=1$. Since $d_1=(1-t)M_n$, choosing $M_0=t>0$, the sequence $\{t, M_1, M_2, M_3, \ldots\}$ is the maximal parameter sequence of $\{d_n\}_{n=1}^{\infty}$. Since $t > 0$, $\{d_n\}_{n=1}^{\infty}$ is a~non SPPCS and hence, by Lemma~\ref{lem:criteria2_sppcs_intermsof_comp_chn_seq} the sequence $\{a_n\}_{n=1}^{\infty}$ is a SPPCS so that $\big\{k_n^{(t)}\big\}_{n=0}^{\infty}$ is also its maximal parameter sequence. Thus, by results established in~\cite{Ranga_chn}, the measure $\nu^{(t)}(z)$ does not have a pure point at $z=1$. This proves the f\/irst part of the theorem.

Now to prove the second part, we f\/irst have
\begin{gather*}
\beta_{n-1}^{(t)} = \overline{\tau}_n \left[\frac{1-2k_n^{(t)}-ic_{n}}{1+ic_{n}}\right]=
\overline{\tau}_n \left[\frac{-1+2m_n^{(t)}-ic_{n}} {1+ic_{n}}\right].
\end{gather*}
By conjugation of the expression for $\alpha_{n-1}^{(t)}$, we have
\begin{gather*}
-\overline{\alpha}_{n-1}^{(t)} = \tau_n \left[\frac{-1+2m_n^{(t)}-ic_{n}}{1-ic_{n}}\right],
\end{gather*}
which leads to the second part of the theorem.

Clearly with $c_n = (-1)^n c$, $n \geq 1$ we have $\overline{\tau}_{2n} = 1$ and $\overline{\tau}_{2n+1} = \frac{1-ic}{1+ic}$. Thus, the third part of
 the theorem is established.

The last part follows by taking $\overline{\tau}_n \overline{\tau}_{n-1} = 1$, $n \geq 1$. This is only possible if $c_n = 0$, $n \geq 1$.
\end{proof}

The perturbation of the Verblunsky coef\/f\/icients in case of OPUC and of the recurrence coef\/f\/icients in case of the real line play an important role in the spectral theory of orthogonal polynomials. The reader is referred to \cite{GARZA} and \cite{FM-Pert} for some details. For a recent work in this direction, we refer to~\cite{Paco_co}.

The last two parts of Theorem~\ref{thm:comp_chn_seq_as_perturbation} are important cases of Aleksandrov transformation and, in the case of last part
 gives rise to second kind polynomials for the measure $\mu^{(t)}$~\cite{Simon1}. In this particular case, the recurrence relation~\eqref{eqn:POP_ttrr_c_n-d_n} assumes a very simple form, similar to that considered in~\cite{DG}.

In the next section, starting with particular minimal parameter sequences and assuming $c_n=0$, $n\geq1$, we construct the para-orthogonal polynomials and the related Szeg\H{o} polynomials to illustrate our results.

\section{An illustration involving Carath\'{e}odory functions}\label{sec-complement_carat}
In a series of papers \cite{JNT-Trig,JNT-Surv, JNT-Mom}, Jones et al.\ during their investigation of the connection between Szeg\H{o} polynomials and continued fractions introduced the following
\begin{gather}\label{eqn:PC_contfrac_def}
%\delta_{0}-\frac{2\delta_0}{1+\frac{1}{\bar{\delta}_1z+\frac{(1-|\delta_1|^2)z} %{\delta_1+\frac{1}{\bar{\delta}_2z+\frac{(1-|\delta_2|^2)z}{\delta_2z+\ddots}}}}}.
\delta_{0}-\frac{2\delta_0}{1}
\begin{array}{cc}\\$+$\end{array}
\frac{1}{\bar{\delta}_1z}
\begin{array}{cc}\\$+$\end{array}
\frac{\big(1-|\delta_1|^2\big)z}{\delta_1}
\begin{array}{cc}\\$+$\end{array}
\frac{1}{\bar{\delta}_2z}
\begin{array}{cc}\\$+$\end{array}
\frac{\big(1-|\delta_2|^2\big)z}{\delta_2z}
\begin{array}{cc}\\$+$\end{array}
\cdots.
\end{gather}
These are called Hermitian Perron--Carath\'eodory fractions or HPC-fractions and are also used to solve the trigonometric moment problem. They are completely determined by $\delta_n\in\mathbb{C}$, where $\delta_0\neq0$ and $|\delta_n|\neq1$ for $n\geq1$. Under the stronger conditions $\delta_0>0$ and $|\delta_n|<1$, for $n\geq1$, \eqref{eqn:PC_contfrac_def} is called a positive PC fraction (PPC-fractions). Let $\mathcal{P}_n(z)$ and $\mathcal{Q}_n(z)$ be respectively the
numerator and denominator of the $n^{th}$ approximant of a PPC-fraction where $\mathcal{Q}_n(z)$ is a~polynomial of degree~$n$ and $\mathcal{P}_n(z)$ of degree at most $n$. Then \cite[Theorems~3.1 and~3.2]{JNT-Mom} $\Phi_n(z)$ are precisely the odd ordered denominators $\mathcal{Q}_{2n+1}(z)$ and $\Phi_{n}^{*}(z)$ the even ordered denomina\-tors~$\mathcal{Q}_{2n}(z)$. The $\delta_n's$ are then given by $\delta_n=\Phi_n(0)$ and are called the Schur parameters or the ref\/lection coef\/f\/icients. This gives the following equivalent set of recurrence relations for the Szeg\H{o} polynomials:
\begin{gather*}
\Phi^{*}_{n}(z) = \bar{\delta}_{n}z \Phi_{n-1}{z}+\Phi_{n-1}^{*}(z),\\
\Phi_{n}(z) = \delta_n\Phi^{*}_{n}(z)+\big(1-|\delta_n|^2\big)z \Phi_{n-1}(z), \qquad n\geq1.
\end{gather*}
Further, if \eqref{eqn:PC_contfrac_def} is a positive PC-fraction, there exists a pair of formal power series
\begin{gather*}
\mathcal{L}_0=\mu_0+2\sum_{k=1}^{\infty}\mu_kz^k, \qquad \mathcal{L}_{\infty}=-\mu_0-2\sum_{k=1}^{\infty}\mu_{-k}z^{-k},
\end{gather*}
where $\mu_k$ are the moments as def\/ined earlier and such that
\begin{gather*}
\mathcal{L}_0-\Lambda_0\left(\frac{\mathcal{P}_{2n}} {\mathcal{Q}_{2n}}\right)=\mathcal{O}\big(z^{n+1}\big),
\qquad \mathcal{L}_{\infty}-\Lambda_{\infty}\left(\frac{\mathcal{P}_{2n+1}} {\mathcal{Q}_{2n+1}} \right)=\mathcal{O}\left(\frac{1}{z^{n+1}}\right).
\end{gather*}
Here, $\Lambda_0(\mathcal{R}(z))$ and $\Lambda_{\infty}(\mathcal{R}(z))$ are the Laurent series expansion of the rational function~$\mathcal{R}(z)$ about~0 and $\infty$ respectively. For details regarding correspondence of continued fractions to power series, see~\cite{JT,Lisa}.

For $|\zeta|<1$, the polynomials
\begin{gather*}
\Psi_n(z)=\int_{\partial\mathbb{D}}\frac{z+\zeta}{z-\zeta}(\Phi_n(z)-\Phi_n(\zeta))d\mu(\zeta), \qquad n\geq1,
\end{gather*}
are known in literature as the associated Szeg\H{o} polynomials or polynomials of the second kind~\cite{Gero}. They arise as the odd ordered numerators of
\eqref{eqn:PC_contfrac_def}. The function $-\Psi^{*}_{n}(z)$ is called the polynomial associated with $\Phi_{n}^{*}(z)$ and are the even ordered numerators in~\eqref{eqn:PC_contfrac_def}. It is also known that for $|z|<1$, there exists a function $\mathcal{C}(z)=\int_{\partial\mathbb{D}} \frac{\zeta+z}{\zeta-z}d\mu(\zeta)$ with $\operatorname{Re} \mathcal{C}(z)>0$ such that
\begin{gather*}
\mathcal{C}(z)- \frac{\Psi^{*}_{n}(z)}{\Phi^{*}_{n}(z)}= \mathcal{O}\big(z^{n+1}\big).
\end{gather*}
$\mathcal{C}(z)$ is called the Carath\'{e}odory function associated with the PPC-fraction~\eqref{eqn:PC_contfrac_def} or with the Szeg\H{o} polynomials $\Phi_{n}(z)$ obtained from this PPC-fraction. The ratio $\Psi_n(z)/\Phi_n(z)$ also converges to a~function $\hat{\mathcal{C}}(z)$ called the Carath\'{e}odory reciprocal of $\mathcal{C}(z)$~\cite{JNT-Surv} and is def\/ined by
\begin{gather*}
\mathcal{C}(z)= -\overline{\hat{\mathcal{C}}(1/\bar{z})}.
\end{gather*}
The convergence is uniform on compact subsets of $|z|<1$ and $|z|>1$ respectively. Also, $\mathcal{L}_0$ is the Taylor series expansion of $\mathcal{C}(z)$ about~0 and $\mathcal{L}_{\infty}$ is that of $\hat{\mathcal{C}}(z)$ about~$\infty$.

Consider the sequence $\{\delta_n\}_{n=1}^{\infty}$, which satisf\/ies $\delta_0>0$, $|\delta_n|<1$ and
\begin{gather}\label{eqn:schur_para_definition}
\delta_{n+1}-\delta_n=\delta_n\delta_{n+1}, \qquad n\geq1.
\end{gather}
Our aim in this section is to use a chain sequence to construct the Szeg\H{o} polynomials $\Phi_n^{(t)}(z)$, having $\delta_n\in\mathbb{R}$ and satisfying~\eqref{eqn:schur_para_definition} as the Verblunsky coef\/f\/icients. We will also use the complementary chain sequence to get another sequence of
Szeg\H{o} polynomials $\tilde{\Phi}_n^{(t)}(z)$ which has $-\delta_n$ as the Verblunsky coef\/f\/icients. The associated Carath\'{e}odory function in each case is also given and it is shown that there exists a relation between them.

We start with the sequence $\big\{m_n^{(t)}\big\}_{n=0}^{\infty}$, where $m_0^{(t)}=0$ and $m_n^{(t)}=(1-\delta_n)/2$, $n\geq1$. These minimal parameters are obtained by f\/irst substituting $c_k=0$, $k\geq1$ in the Verblunsky coef\/f\/icients~\eqref{eqn:gen_verb_coeff} and then equating them to $\delta_n$. The corresponding chain sequence is
\begin{gather*}%\label{eqn:schur_para_chain_seq}
d_1=\frac{1-\delta_1}{2} \qquad\mbox{and}\qquad d_n=\frac{1}{4}(1+\delta_{n-1})(1-\delta_n)= \frac{1}{4}(1-2\delta_{n-1}\delta_n), \qquad n\geq2.
\end{gather*}
The following are two algebraic relations of $\delta_n$ which will be needed later and can be proved by simple induction using~\eqref{eqn:schur_para_definition}.
\begin{gather*}
\delta_1\delta_2+\delta_2\delta_3+\delta_3\delta_4+ \cdots+\delta_n\delta_{n+1}=\delta_{n+1}-\delta_1, \qquad n\in\mathbb{N}.
\end{gather*}
and
\begin{gather}\label{eqn:schur_para_alg_rel2}
\delta_n=\frac{\delta_{n+1}}{1+\delta_{n+1}}= \cdots= \frac{\delta_{n+k}}{1+k\delta_{n+k}}, \qquad k\in\mathbb{N}.
\end{gather}

\begin{Proposition}The monic polynomial
\begin{gather}\label{eqn:monic POP_ronning_ccs}
R_n(z)=1+\sum_{k=1}^{n}[1+2k(n-k)\delta_1\delta_n]z^k
\end{gather}
satisfies the recurrence relation
\begin{gather*}%\label{eqn:monic_POP_ronning_ccs_ttrr}
R_{n+1}(z)=(z+1)R_{n}(z)-(1-2\delta_n\delta_{n+1})zR_{n-1}(z), \qquad n\geq1,
\end{gather*}
with the initial conditions, $R_{0}(z)=1$ and $R_{1}(z)=z+1$.
\end{Proposition}

\begin{proof} First, note that $R_{1}(z)$ given by~\eqref{eqn:monic POP_ronning_ccs} satisf\/ies the initial condition. Suppose $R_n(z)$ has this form and satisf\/ies the recurrence relation for $n=1,2,\dots,j$. We shall now show
\begin{gather}\label{eqn:complement_carat_ttrr_monic_poly}
R_{j+1}(z)+(1-2\delta_j\delta_{j+1})zR_{j-1}(z)=(z+1)R_{j}(z).
\end{gather}
Using \eqref{eqn:schur_para_alg_rel2}, the coef\/f\/icient of $z^k$ in the left-hand side of~\eqref{eqn:complement_carat_ttrr_monic_poly} is
\begin{gather}
1+2k(j-k+1)\delta_1\delta_{j+1}+ (1-2\delta_j\delta_{j+1})[1+2(k-1)(j-k)\delta_1\delta_{j-1}]\nonumber\\
\qquad{} =1+2\frac{k(j-k+1)}{j}(\delta_{j+1}-\delta_1)+1- 2(\delta_{j+1}-\delta_j)+2\frac{(k-1)(j-k)}{j-2} (\delta_{j-1}-\delta_1)\nonumber\\
\qquad\quad{} -\frac{2\cdot2(k-1)(j-k)}{j-2}(\delta_{j-1}-\delta_1)(\delta_{j+1}-\delta_j).\label{eqn:coeff_z^k_proof}
\end{gather}
It is easy to verify that the coef\/f\/icients of $\delta_{j+1}$ and $\delta_{j-1}$ vanish in~\eqref{eqn:coeff_z^k_proof}. The coef\/f\/icient of $\delta_1$ is
\begin{gather}
-\frac{2k(j-k+1)}{j}-\frac{2(k-1)(j-k)}{j-2}- \frac{2\cdot2(k-1)(j-k)}{j(j-2)}+\frac{2\cdot2(k-1) (j-k)}{(j-1)(j-2)}\nonumber\\
 \qquad{} =-\frac{2k(j-k)}{j-1}-\frac{2(k-1)(j-k+1)}{j-1}.\label{eqn:coeff_z^1_proof}
\end{gather}
Similarly, the coef\/f\/icient of $\delta_j$ is
\begin{gather}
2+\frac{2\cdot2(k-1)(j-k)}{j-1} =\frac{2k(j-k)}{j-1}+\frac{2(k-1)(j-k+1)}{j-1}.\label{eqn:coeff_delta_j_proof}
\end{gather}
Using \eqref{eqn:coeff_z^1_proof} and \eqref{eqn:coeff_delta_j_proof} in~\eqref{eqn:coeff_z^k_proof}, the coef\/f\/icient of~$z^k$ in the left-hand side of~\eqref{eqn:complement_carat_ttrr_monic_poly} is given by
\begin{gather*}
[1+2(k-1)(j-k+1)\delta_1\delta_j]+ [1+2k(j-k)\delta_1\delta_j],
\end{gather*}
which is nothing but the coef\/f\/icient of~$z^k$ in the right-hand side of~\eqref{eqn:complement_carat_ttrr_monic_poly}. Hence, by induction the proof is complete.
\end{proof}

We now obtain the Szeg\H{o} polynomials $\Phi_n^{(t)}(z)$ from the para-orthogonal polynomials $R_n(z)$ given by~\eqref{eqn:monic POP_ronning_ccs}. Using~\eqref{eqn:szego_from_gen_verb_coeff} and~\eqref{eqn:monic POP_ronning_ccs}, it can be seen that the coef\/f\/icient of~$z^k$, $1\leq k\leq n-1$, in~$\Phi_{n}^{(t)}(z)$ is~$-\delta_{n}(1-2k\delta_1)$. Hence, the Szeg\H{o} polynomials are given by
\begin{gather}\label{eqn:szego_poly_ronning_ccs}
\Phi_n^{(t)}(z)=z^n-\delta_n\big[(1-2(n-1)\delta_1)z^{n-1}+ \cdots+(1-2\delta_1)z+1\big],\qquad n\geq1,
\end{gather}
with $\alpha_{n-1}^{(t)}=-\delta_n$.

We now give the Carath\'{e}odory function associated with the parameters~$\delta_n$'s given by~\eqref{eqn:schur_para_definition}. Consider
\begin{gather*}
\mathcal{C}(z)= 1-\frac{2(1-\sigma)z}{1+(1-2\sigma)z}= \frac{1-z}{1+(1-2\sigma)z}, \qquad |z|<1,
\end{gather*}
where $0<\sigma<1$. That $\mathcal{C}(z)$ corresponds to a PPC-fraction with the parameter~$\gamma_n$, where
\begin{gather}\label{eqn:schur_para_ronning_ccs}
\gamma_n=\frac{1}{n+\frac{\sigma}{1-\sigma}}, \qquad n\geq1.
\end{gather}
can be shown by applying the algorithm~\cite{JNT-Surv} which is similar to the Schur algorithm. With the initial values $\mathcal{C}_0(z)=(1-z)/(1+(1-2\sigma)z)$, $\gamma_0=\mathcal{C}_0(0)=1$, def\/ine
\begin{gather*}
\mathcal{C}_1(z)=\frac{\gamma_0-\mathcal{C}_0(z)} {\gamma_0+\mathcal{C}_0(z)}, \qquad \gamma_1=\mathcal{C}'_1(0).
\end{gather*}
Then
\begin{gather*}
\mathcal{C}_1(z)=\frac{z}{1+\frac{\sigma}{1-\sigma}- \left(1-\frac{1-2\sigma}{1-\sigma}\right)z},\qquad\mbox{and}\qquad \gamma_1=\frac{1}{1+\frac{\sigma}{1-\sigma}}.
\end{gather*}
Assume for $k\geq1$ the following
\begin{gather*}
\mathcal{C}_k(z)=\frac{z}{k+\frac{\sigma}{1-\sigma}- \left(k-\frac{1-2\sigma}{1-\sigma}\right)z},\qquad \gamma_k=\mathcal{C}'_k(0).
\end{gather*}
This is true for $k=1$. Now def\/ine
\begin{gather}\label{eqn:carat_schur_algo}
\mathcal{C}_{k+1}(z)=\frac{\gamma_kz-\mathcal{C}_k(z)} {\gamma_k\mathcal{C}_k(z)-z}, \qquad n\geq1.
\end{gather}
It can be shown that
\begin{gather*}
\gamma_k=\frac{1-\sigma}{k-(k-1)\sigma}= \frac{1}{k+\frac{\sigma}{1-\sigma}},
\end{gather*}
which is also true for $k=1$. Simplifying \eqref{eqn:carat_schur_algo}, we obtain
\begin{gather*}
\mathcal{C}_{k+1}= \frac{z}{\left(k+1+\frac{\sigma}{1-\sigma}\right)- \left(k+1-\frac{1-2\sigma}{1-\sigma}\right)z},
\end{gather*}
from which $\gamma_{k+1}=\frac{1}{k+1+\frac{\sigma}{1-\sigma}}$. Hence by induction, \eqref{eqn:schur_para_ronning_ccs} and because of the uniqueness of the Carath\'{e}odory function that corresponds to a given PPC-fraction, the assertion follows. Moreover, observe that $\delta_n=-\gamma_n$ satisf\/ies~\eqref{eqn:schur_para_definition} and so $\Phi_n^{(t)}(0)=\frac{1}{n+\frac{\sigma}{1-\sigma}}$.

From the power series expansion of $\mathcal{C}(z)$, we also obtain the moments as
\begin{gather*}
\mu_0=1,\qquad \mu_k=(-1)^k(1-\alpha)(1-2\alpha)^{k-1}, \qquad k\geq1.
\end{gather*}
Using the fact that the Verblunsky coef\/f\/icients are all real, from \eqref{eqn:szego_norm}, we have
\begin{gather*}
\chi_{n}^{-2}=\prod_{k=1}^{n}\big(1-\delta_k^2\big).
\end{gather*}
Further
\begin{gather*}
\delta_n=\frac{1}{n+\frac{\sigma}{1-\sigma}}= \frac{1-\sigma}{n(1-\sigma)+\sigma}, \qquad n\geq1,
\end{gather*}
and we obtain
\begin{gather*}
\begin{split}
& 1-\delta_n^2 =\frac{[n(1-\sigma)+\sigma-1+\sigma][n(1-\sigma)+\sigma+1-\sigma]}{[n(1-\sigma)+\sigma]^2}\\
& \hphantom{1-\delta_n^2}{} =\frac{[(n-1)-(n-2)\sigma][(n+1)-n\sigma]}{[n-(n-1)\sigma]^2},
\end{split}
\end{gather*}
which yields the fact that
\begin{gather*}
\chi_n^{-2}=\frac{\sigma[(n+1)-n\sigma]}{[n-(n-1)\sigma]}.
\end{gather*}
Rewriting the right-hand expression as $\sigma\big(1+\frac{1-\sigma}{n(1-\sigma)+\sigma}\big)$ gives
\begin{gather*}
\chi_n^{-2}=\big\|\Phi_n^{(t)}(z)\big\|^2=\sigma(1+\delta_n),
\end{gather*}
which tends to $\sigma>0$ as $n\rightarrow\infty$.

Consider now the parameter sequence $\big\{k_n^{(t)}\big\}_{n=0}^{\infty}$, def\/ined by $k_0^{(t)}=0$ and $k_n^{(t)}=1-m_n^{(t)}=(1+\delta_n)/2$, $n\geq1$. From~\eqref{eqn:schur_para_definition}, it is easy to check that $1+\delta_{n+1}=1/(1-\delta_n)$, $n\geq1$. In this case, the constant sequence $\{1/4\}$
becomes the complementary chain sequence so that equation~\eqref{eqn:POP_ttrr_c_n-d_n} assumes the form
\begin{gather*}
\tilde{R}_{n+1}(z)=[1+z]\tilde{R}_{n}(z)-z\tilde{R}_{n-1}(z), \qquad n\geq1.
\end{gather*}
The polynomials satisfying the above recurrence relation are the palindromic polynomials $z^n+\lambda(z^{n-1}+\cdots+z)+1$. For $\lambda=1$, the para-orthogonal polynomials are the partial sums of the geometric series given by
\begin{gather*}
\tilde{R}_{n}(z)= 1+z+z^2+\cdots+z^n=\frac{1-z^{n+1}}{1-z}, \qquad n\geq1.
\end{gather*}
Then \eqref{eqn:szego_from_gen_verb_coeff} yields the Szeg\H{o} polynomial
\begin{gather}\label{eqn:szego_ronning_case}
\tilde{\Phi}_n^{(t)}(z)= z^n+\delta_{n}z^{n-1}+\cdots+\delta_nz+\delta_n, \qquad n\geq1,
\end{gather}
with $\alpha_{n-1}^{(t)}=-\delta_n$. The polynomials $\tilde{\Phi}_n^{(t)}(z)$ have been considered in~\cite{Ronn} where it is proved that
\begin{gather}\label{eqn:schur_para_ronning_case}
\tilde{\Phi}_n^{(t)}(0)= \delta_n=-\frac{1}{n+\frac{\sigma}{1-\sigma}}, \qquad n\geq1.
\end{gather}
Further, the corresponding Carath\'{e}odory function is $\tilde{\mathcal{C}}(z)=\frac{1+(1-2\sigma)z}{1-z}$, $|z|<1$, where $0<\sigma<1$. This is a special case when all the moments are equal to $\tilde{\mu}=(1-\sigma)$. We summarize the above facts as a theorem.

\begin{Theorem}Consider the real sequence $\{\delta_n\}_{n=0}^{\infty}$ satisfying $\delta_n-\delta_{n-1}=\delta_{n-1}\delta_n$, $n\geq1$ under the restrictions $\delta_0>0$ and $|\delta_n|<1$, $n\geq1$. If $\mathcal{C}(z)$ is a Carath\'{e}odory function whose PPC-fraction can be obtained from the minimal parameter sequence $\{m_n\}$, where $2m_n=1-\delta_n$, $n\geq1$, then $1-m_n$ gives the PPC-fraction corresponding to the Carath\'{e}odory function $1/\mathcal{C}(z)$.
\end{Theorem}

Note that an equivalent statement using Schur parameters is given in~\cite{qRonn}. Further, let~$\mu^{(t)}(z)$ be the probability measure associated with the positive chain sequence $\{d_n\}_{n=1}^{\infty}$. Since its complementary chain sequence $\{1/4\}$ is not a SPPCS, by Lemma~\eqref{lem:criteria2_sppcs_intermsof_comp_chn_seq} $\{d_n\}_{n=1}^{\infty}$ is a SPPCS and hence $\mu^{(t)}(z)$ has zero jump $(t=0)$ at $z=1$. If $\nu^{(t)}(z)$ is the measure associated with $\{1/4\}$, $\nu^{(t)}(z)$ has a~jump $t=1/2$ at $z=1$. Finally as shown in~\cite{Ronn}, $\nu^{(1/2)}(\theta)$ is of the form,
\begin{gather*}
d\nu^{(1/2)}(\theta)= d\nu_{s}^{(1/2)}(\theta)+(1-\tilde{\mu})d(\theta),
\end{gather*}
where $d\nu_{s}^{(1/2)}(\theta)$ is a point measure with mass $\tilde{\mu}$ at $z$=1 and mass zero elsewhere.

We end this illustration with two observations which we state as remarks.

\begin{Remark}
Suppose the minimal parameters are given in terms of some variable $\varepsilon$. It follows that the coef\/f\/icients of the polynomial $R_n(z)$
satisfying \eqref{eqn:POP_ttrr_c_n-d_n} with $c_n=0$ for $n\geq1$ will be given in terms of $\varepsilon$. Since, it is clear that $R_n(z)$ is palindromic for the chain sequence $\{d_n\}=\{1/4\}$, $R_n(z)$ can always be expressed as the sum of two polynomials, one of them being a palindromic and the other one being such that it vanishes whenever $\varepsilon$ is chosen so that $d_n=1/4$.
\end{Remark}

\begin{Remark}\label{remark:limiting case for ronning parameter}
As $n\rightarrow\infty$, both the minimal parameter sequences approach~1/2. From the expressions~\eqref{eqn:szego_poly_ronning_ccs} and~\eqref{eqn:szego_ronning_case} it is clear that for f\/ixed~$z$, $\Phi_{n}^{(t)}(z)$ and $\tilde{\Phi}_{n}^{(t)}(z)$ approach~$z^n$ as~$n$ becomes large. The polynomials $z^n$ are called the Szeg\H{o}--Chebyshev polynomials and correspond to the standard Lebesgue measure on the unit circle.
\end{Remark}

\section{An illustration using Gaussian hypergeometric functions}\label{sec-complement_gauss_hyper}
The Gaussian hypergeometric function, with the complex parameters $a$, $b$ and $c$ is def\/ined by the power series
\begin{gather*}
F(a,b;c;z)=\sum_{n=0}^{\infty}\frac{(a)_n(b)_n}{(c)_n(1)_n}z^n, \qquad |z|<1,
\end{gather*}
where $c\neq0,-1,-2,\dots$ and $(a)_n$ is the Pochhammer symbol. With specialized values of the parameters $a$, $b$ and $c$, many elementary functions can be
represented by the Gaussian hypergeometric functions or their ratios. If $\operatorname{Re}(c-a-b)>0$, the series converges for $|z|=1$ to the value given by
\begin{gather*}
F(a,b;c;1)=\sum_{k=0}^{\infty}\frac{(a)_k(b)_k}{(c)_kk!}= \frac{\Gamma(c)\Gamma(c-a-b)}{\Gamma(c-a)\Gamma(c-b)}.
\end{gather*}
In case the series is terminating, we have the Chu--Vandermonde identity~\cite{AAR}
\begin{gather}\label{eqn:chu_vandermonde_identity}
F(-n,b;c;1)=\frac{(c-b)_n}{(c)_n}.
\end{gather}
Two hypergeometric functions $F(a_1,b_1;c_1;z)$ and $F(a_2,b_2;c_2,z)$ are said to be contiguous if the dif\/ference between the corresponding parameters
is at most unity. A~linear combination of two contiguous hypergeometric functions is again a hypergeometric function. Such relations are called contiguous relations and have been used to explore many hidden properties of the hypergeometric functions, for example by Gauss who found continued fraction expansions for ratios of hypergeometric functions~\cite{Rama} and hence for the special functions that these ratios represent. In some special cases, the contiguous
relations can also be related to the recurrence relations for orthogonal polynomials. Consider one such relation~\cite{AAR}
\begin{gather*}%\label{eqn:ranga_szego_conti_reln}
(c-a)F(a-1,b;c;z)=(c-2a-(b-a)z)F(a,b;c;z)+ a(1-z)F(a+1,b;c;z),
\end{gather*}
which as shown in~\cite{Ranga_sze}, can be transformed to the three term recurrence relation
\begin{gather}\label{eqn:ranga_szego_conti_reln_ttrr}
\varrho_{n+1}(z)=\left(z+\frac{c-b+n}{b+n}\right)\varrho_n(z)- \frac{n(c+n-1)}{(b+n-1)(b+n)}\varrho_{n-1}(z), \qquad n\geq1,
\end{gather}
satisf\/ied by the monic polynomial
\begin{gather}\label{eqn:ranga_szego_conti_reln_polynomial}
\varrho_n(z)=\frac{(c)_n}{(b)_n}F(-n,b;c;1-z).
\end{gather}
It was also shown that for the specif\/ic values $b=\lambda\in\mathbb{R}$ and $c=2\lambda-1$, the polynomials~\eqref{eqn:ranga_szego_conti_reln_polynomial} are Szeg\H{o} polynomials. We note that with $b = \lambda+1$, $\varrho_n(z)$ given by~\eqref{eqn:ranga_szego_conti_reln_polynomial} are called the circular Jacobi polynomials \cite[Example~8.2.5]{Ismail}. For other specialized values of~$b$ and~$c$ in~\eqref{eqn:ranga_szego_conti_reln_ttrr}, $\varrho_n(z)$~also becomes the para-orthogonal polynomial.

Let $\lambda>-1/2\in\mathbb{R}$. Taking $b=\lambda+1$ and $c=2\lambda+2$, \eqref{eqn:ranga_szego_conti_reln_ttrr} reduces to
\begin{gather*}
\varrho_{n+1}(z)=(z+1)\varrho_{n}(z)- \frac{n(2\lambda+n+1)}{(\lambda+n)(\lambda+n+1)} z\varrho_{n-1}(z), \qquad n\geq1,
\end{gather*}
satisf\/ied by
\begin{gather*}
\varrho_{n}(z)=R_n(z)= \frac{(2\lambda+2)_n}{(\lambda+1)_n} F(-n,\lambda+1;2\lambda+2;1-z), \qquad n\geq1.
\end{gather*}
Consider now the sequence $\{d_{n+1}\}_{n=1}^{\infty}$, where
\begin{gather*}
d_{n+1}= \frac{1}{4}\frac{n(2\lambda+n+1)}{(\lambda+n)(\lambda+n+1)}, \qquad n\geq1.
\end{gather*}
As established in \cite[Example~3]{Swami}, for $\lambda>-1$, the sequence $\{d_{n+1}\}_{n=1}^{\infty}$ is a positive chain sequence and $\{\mathfrak{m}_n\}_{n=0}^{\infty}$, where
\begin{gather*}
\mathfrak{m}_{n} = \frac{n}{2(\lambda+n+1)}, \qquad n \geq 0,
\end{gather*}
is its minimal parameter sequence. When $-1/2 \geq \lambda > -1$, $\{\mathfrak{m}_{n}\}_{n=0}^{\infty}$ is also the maximal parameter sequence of $\{d_{n+1}\}_{n=1}^{\infty}$, which makes it a SPPCS. However, when $\lambda > -1/2$ then $\{d_{n+1}\}_{n=1}^{\infty}$ is not a~SPPCS and its maximal parameter sequence $\{M_{n+1}\}_{n=0}^{\infty}$ is such that
\begin{gather*}
M_{n+1} = \frac{2\lambda+n+1}{2(\lambda+n+1)}, \qquad n \geq 0.
\end{gather*}
The coef\/f\/icients $d_{n+1}$, $n \geq 1$ are the same coef\/f\/icients occurring in the recurrence formula for ultraspherical (or Gegenbauer) polynomials.

Further, for $\lambda > -1/2$ and $0 \leq t < 1$, if $\big\{m_n^{(t)}\big\}_{n=0}^{\infty}$ is the minimal parameter sequence of the positive chain sequence $\{d_n\}_{n=1}^{\infty}$, obtained by adding $d_{1} = (1-t)M_{1} $ to $\{d_{n+1}\}_{n=1}^{\infty}$, then from~\eqref{eqn:szego_from_gen_verb_coeff}
\begin{gather*}
\Phi_n^{(t)}(z)=R_{n}(z)-2\big(1-m_{n}^{(t)}\big)R_{n-1}(z), \qquad n\geq1
\end{gather*}
and are the monic OPUC with respect to the measure $\mu^{(t)}(z)$, where $\mu^{(t)}(z)$ is as def\/ined by~\eqref{eqn:relation between measure with jump t and measure with zero jump}. To f\/ind $\mu^{(t)}(z)$, we f\/irst f\/ind the measure $\mu^{(0)}(z)$ arising when $\{d_n\}_{n=1}^{\infty}$ becomes a SPPCS $(t=0)$. As shown in~\cite{Ranga_sze}, the monic OPUC are given by
\begin{gather*}
\Phi_n^{(0)}(z) =R_{n}(z)-2(1-M_{n})R_{n-1}(z) =\frac{(2\lambda+1)_n}{(\lambda+1)_n}F(-n,\lambda+1;2\lambda+1;1-z), \qquad n\geq1.
%\label{eqn:ranga_szego_polynomials}
\end{gather*}
Using the identity \eqref{eqn:chu_vandermonde_identity}, the Verblunsky coef\/f\/icients are given by
\begin{gather}\label{eqn:verblunsky coeff for ranga szego lambda}
\alpha_{n-1}^{(0)}=-\Phi_{n}^{(0)}(0)=- \frac{(\lambda)_n}{(\lambda+1)_n}, \qquad n\geq1.
\end{gather}
The Verblunsky coef\/f\/icients $\alpha_{n-1}^{(0)}$ are associated with the non-trivial probability measure given by~\cite{Ranga_sze}
\begin{gather*}
d\mu^{(0)}\big(e^{i\theta}\big)= \tau^{(\lambda)}\sin^{2\lambda}(\theta/2)d\theta,
\end{gather*}
where
\begin{gather*}
\tau^{(\lambda)}= \frac{|\Gamma(1+\lambda)|^2}{\Gamma(2\lambda+1)}4^\lambda.
\end{gather*}
Hence
\begin{gather*}
\int_{\partial\mathbb{D}}f(\zeta)d\mu^{(t)}(\zeta)= (1-t)\tau^{(\lambda)}\int_{0}^{2\pi}f\big(e^{i\theta}\big) \sin^{2\lambda}(\theta/2)d\theta+tf(1).
\end{gather*}

Further characterization of Szeg\H{o} polynomials is provided below as it is not possible to f\/ind closed form expressions for the coef\/f\/icients of the para-orthogonal polynomials and Szeg\H{o} polynomials. Since $\{R_n(z)\}$, depends on the parameter $b$ $(=\lambda+1)$, in what follows, we denote~$R_n(z)$ by~$R_n^{(b)}(z)$. We also denote~$c_n$ and~$d_{n}$ by~$c_n^{(b)}$ and~$d_n^{(b)}$ respectively. Now, note that if
\begin{gather*}
Q_n^{(b)}(z) = \frac{1}{2(1-t)M_1}\int_{\mathbb{T}} \frac{R_n^{(b)}(z) - R_n^{(b)}(\zeta)} {z - \zeta} (1-\zeta)d \mu^{(t)}(\zeta), \qquad n \geq 0,
\end{gather*}
then $\big\{Q_n^{(b)}(z)\big\}_{n=0}^{\infty}$ satisf\/ies
\begin{gather*}
Q_{n+1}^{(b)}(z) = \big[\big(1+ic_{n+1}^{(b)}\big)z + \big(1-ic_{n+1}^{(b)}\big)\big] Q_{n}^{(b)}(z) - 4 d_{n+1}^{(b)} z Q_{n-1}^{(b)}(z), \qquad n \geq 1,
\end{gather*}
with $Q_{0}^{(b)}(z) = 0$ and $Q_{1}^{(b)}(z) = 1$. That is, the three term recurrence for $\big\{Q_n^{(b)}(z)\big\}_{n=0}^{\infty}$ is the same as for $\big\{R_n^{(b)}(z)\big\}_{n=0}^{\infty}$, with the dif\/ference being only on the initial conditions. The polynomials $\big\{Q_n^{(b)}(z)\big\}$ are generally called the numerator polynomials associated with $\big\{R_n^{(b)}(z)\big\}$. Further, observe that the three term recurrence for $\big\{Q_n^{(b)}(z)\big\}_{n=0}^{\infty}$ can also be given in the shifted form
\begin{gather}\label{Eq-TTRR-Qn(b)}
 Q_{n+2}^{(b)}(z) = \big[\big(1+ic_{n+2}^{(b)}\big)z + \big(1-ic_{n+2}^{(b)}\big)\big] Q_{n+1}^{(b)}(z) - 4 d_{n+2}^{(b)}z Q_{n}^{(b)}(z), \qquad n \geq 1,
\end{gather}
with $Q_{1}^{(b)}(z) = 1$ and $Q_{2}^{(b)}(z) = \big(1+ic_{2}^{(b)}\big)z + \big(1-ic_{2}^{(b)}\big)$.

Consider now the parameter sequence given by $k_n^{(t)}=1-m_n^{(0)} = n/[2(\lambda+n)]$ for $n\geq1$. For sake of clarity, we would like to note that $t$ need not be necessarily~0. It depends on whether the resulting chain sequence for $\big\{k_n^{(t)}\big\}$, given by
\begin{gather}\label{eqn:ccs-hypergeometric-case}
a_1^{(b)}=\frac{1}{2\lambda+2} \qquad\mbox{and}\qquad a_{n+1}^{(b)} =\frac{1}{4}\frac{(n+1)(2\lambda+n)} {(\lambda+n)(\lambda+n+1)}, \qquad n\geq 1,
\end{gather}
is a SPPCS or not.

Let $\nu^{(t)}(z)$ be the measure associated with the Verblunsky coef\/f\/icients $\big\{\beta_{n-1}^{(t)}\big\}_{n=1}^{\infty}$ given by
\begin{gather*}
\beta_{n-1}^{(t)} = \overline{\tau}_n \left[\frac{1-2k_n^{(t)}-ic_{n}^{(b)}}{1+ic_{n}^{(b)}}\right], \qquad n \geq 1.
\end{gather*}
Following Theorem \ref{thm:comp_chn_seq_as_perturbation}, the corresponding OPUC are
\begin{gather*}
\tilde{\Phi}_n^{(t)}(z)=\frac{\tilde{R}_n^{(b)}(z)- 2\big(1-k_n^{(t)}\big)\tilde{R}_{n-1}^{(b)}(z)} {\prod\limits_{k=1}^{n}\big(1+ic_k^{(b)}\big)}, \qquad n\geq1,
\end{gather*}
where the polynomials $\tilde{R}_n^{(b)}$ are given by
\begin{gather}\label{Eq-TTRR-tildeRn(b)}
 \tilde{R}_{n+1}^{(b)}(z) = \big[\big(1+ic_{n+1}^{(b)}\big)z + \big(1-ic_{n+1}^{(b)}\big)\big] \tilde{R}_{n}^{(b)}(z) - 4 a_{n+1}^{(b)} z \tilde{R}_{n-1}^{(b)}(z), \qquad n\geq 1,
\end{gather}
with $\tilde{R}_{0}^{(b)}(z) = 1$ and $\tilde{R}_{1}^{(b)}(z)=\big(1+ic_{1}^{(b)}\big)z+ \big(1-ic_{1}^{(b)}\big)$. Observing that $c_{n}^{(b)}=c_{n+1}^{(b-1)}$, $a_{n+1}^{(b)}=d_{n+2}^{(b-1)}$, $n \geq 1$, we have from~\eqref{Eq-TTRR-Qn(b)} and~\eqref{Eq-TTRR-tildeRn(b)}
\begin{gather*}
 \tilde{R}_{n}^{(b)}(z) = Q_{n+1}^{(b-1)}(z), \qquad n \geq 0,
\end{gather*}
and thus
\begin{gather*}
 \tilde{\Phi}_n^{(t)}(z)= \frac{Q_{n+1}^{(b-1)}(z)-2\big(1-k_n^{(t)}\big)Q_{n}^{(b-1)}(z)} {\prod\limits_{k=1}^{n}\big(1+ic_{k+1}^{(b-1)}\big)}, \qquad n\geq1.
\end{gather*}
That is, if $R_n^{(b)}(z)$ generates the OPUC $\Phi_n^{(t)}(z)$, $Q_n^{(b-1)}(z)$, which are the numerator polynomials for $R_n^{(b-1)}(z)$ generates the OPUC $\tilde{\Phi}_{n}^{(t)}(z)$ associated with the complementary chain sequences. We note that, in the present case too, $c_n^{(b)} (= c_n) =0$, $n\geq1$ and so by Theorem~\ref{thm:comp_chn_seq_as_perturbation} $\beta_{n-1}^{(t)}=-\alpha_{n-1}^{(0)}$ for $n\geq1$. Hence $d\nu^{(t)}(z)$ are the Aleksandrov measures associated with $d\mu^{(0)}(z)$~\cite{Simon1}.

Further, we note that such Szeg\H{o} polynomials result from perturbations of the Verblunsky coef\/f\/icients obtained in Section~\ref{sec-complement_carat}. Indeed, for $\sigma=\lambda/(1+\lambda)$, $\{\lambda\delta_n\}$ corresponds to the Verblunsky coef\/f\/icients given by~\eqref{eqn:verblunsky coeff for ranga szego lambda}, wheras by Verblunsky theorem, $\{\lambda\gamma_n\}$ corresponds to those given by the complementary chain sequence $\{a_{n+1}^{(b)}\}$ given by~\eqref{eqn:ccs-hypergeometric-case}. Here $\{\delta_n\}$ and $\{\gamma_n\}$ are the ones chosen respectively by~\eqref{eqn:schur_para_ronning_ccs} and~\eqref{eqn:schur_para_ronning_case}.

Further, when $\big\{a_{n+1}^{(b)}\big\}_{n=1}^{\infty}$ is the constant chain sequence $\{1/4\}$, $\tilde{R}_{n}^{(b)}(z)$ are the palindromic polynomials given by
\begin{gather*}
\tilde{R}_{n}^{(b)}(z)= z^n+\nu^{(\lambda)}\big(z^{n-1}+\cdots+z\big)+1, \qquad n\geq1,
\end{gather*}
where $\nu^{(\lambda)}$ is a constant depending on $\lambda$. Here we study the cases $\lambda=0$ and $\lambda=1$ for which the complementary chain sequence $a_{n+1}^{(b)}=1/4$.

{\bf Case 1}, $\lambda=0$. Let
 \begin{gather*}
 \tilde{R}_{n}^{(b)}(z)=z^n+\nu^{(0)}\big(z^{n-1}+\cdots+z\big)+1, \qquad n\geq1.
 \end{gather*}
The complementary chain sequence is $\{1/2,1/4,1/4,\dots\}$ which is known to be a SPPCS. Hence $\big\{k_n^{(t)}\big\}_{n=0}^{\infty}$ where $k_0^{(t)}=0$, $k_n^{(t)}=1/2$, $n\geq1$ is also the maximal parameter sequence implying that $t=0$ and so
\begin{gather*}
\tilde{\Phi}_{n}^{(0)}(z)=z^n+\big(\nu^{(0)}-1\big)z^{n-1}.
\end{gather*}
For $\nu^{(0)}=1$, $\tilde{\Phi}_n^{(0)}(z)=z^n$ and from Remark~\ref{remark:limiting case for ronning parameter}, $\lambda=0$ can be viewed as the limiting case for the Verblunsky coef\/f\/icients obtained in Section~\ref{sec-complement_carat}. Note that the Verblunsky coef\/f\/icients are~0, as can be verif\/ied from~\eqref{eqn:verblunsky coeff for ranga szego lambda}.

{\bf Case 2}, $\lambda=1$. Let
\begin{gather*}
\tilde{R}_{n}^{(b)}(z)= z^n+\nu^{(1)}\big(z^{n-1}+\cdots+z\big)+1, \qquad n\geq1.
\end{gather*}
The complementary chain sequence is $\{1/4,1/4,1/4,\dots\}$ and $k_0^{(t)}=0$, $k_n^{(t)}=n/2(n+1)$, $n\geq1$. In this case, $t=1/2$ and
\begin{gather*}
\tilde{\Phi}_{n}^{(1/2)}(z)= z^n+\left(\nu^{(1)}-\frac{n+2}{n+1}\right)z^{n-1}-\frac{\nu^{(1)}}{n+1}\big(z^{n-2}+\cdots+z\big)-\frac{1}{n+1}, \qquad n\geq1,
\end{gather*}
so that the Verblunsky coef\/f\/icients are given by $1/(n+1)$. Again it can be verif\/ied from~\eqref{eqn:verblunsky coeff for ranga szego lambda} that the Verblunsky coef\/f\/icients corresponding to $\lambda=1$ are $(1)_n/(2)_n=1/(n+1)$. Finally, for $\nu^{(1)}=0$, $\tilde{R}_n^{(b)}=z^n+1$, which has been
 considered as Example~1 in~\cite{Swami}.

\subsection*{Acknowledgements}
The authors wish to thank the anonymous referees for their constructive criticism that resulted in signif\/icant improvement of the content leading to the f\/inal version. The work of the second author was supported by funds from CNPq, Brazil (grants 475502/2013-2 and 305073/2014-1) and FAPESP, Brazil (grant 2009/13832-9).

\pdfbookmark[1]{References}{ref}
\LastPageEnding

\end{document}